\documentclass[a4paper, reqno,10pt]{amsart}

\usepackage{amsmath,amsthm,amssymb,mathtools,amscd}
\usepackage{latexsym,graphicx,color,indentfirst}
\usepackage[linesnumbered,boxed,norelsize]{algorithm2e}
\usepackage{fancyhdr}
\usepackage{mathrsfs}

\makeatletter \@addtoreset{equation}{section} \makeatother
\newtheorem{theorem}{Theorem}[section]
\newtheorem{lemma}[theorem]{Lemma}
\newtheorem{assumption}[theorem]{Assumption}

\newtheorem{proposition}[theorem]{Proposition}

\newtheorem{remark}[theorem]{Remark}

\def\yd{y^\delta}
\def\D{\mathcal D}
\def\d{\delta}

\def\geq{\geqslant}
\def\ge{\geqslant}

\def\l{\left}
\def\le{\leqslant}
\def\leq{\leqslant}
\def\r{\right}

\def\p{\partial}
\def\d{\delta}

\def\l{\langle}
\def\r{\rangle}

\def\X{\mathcal X}
\def\Y{\mathcal Y}

\def\la{\lambda}

\def\p{\partial}
\def\d{\delta}

\def\l{\langle}
\def\r{\rangle}

\def\d{\delta}
\def\la{\lambda}
\def\X{\mathcal X}
\def\Y{\mathcal Y}

\title[ ]{On the asymptotical regularization with convex constraints  for inverse problems }
\author{ Min Zhong}
\address{School of Mathematics, Southeast University, Nanjing, Jiangsu 210096, P. R. China}
\email{\tt min.zhong@seu.edu.cn}
\author{Wei Wang}
\address{Corresponding author. College of Mathematics, Physics and Information Engineering, Jiaxing University,
Zhejiang 314001, P. R. China}
\email{\tt weiwang@zjxu.edu.cn}

\date{\today}

\begin{document}

\begin{abstract}
In this paper, we consider the asymptotical regularization with convex constraints
for nonlinear ill-posed problems. The method allows to use non-smooth penalty terms, including the
L1-like and the total variation-like penalty functionals, which are significant in reconstructing special features of solutions such as sparsity and piecewise constancy.
 Under certain conditions we give convergence
properties of the methods.
Moreover, we propose Runge-Kutta type methods to discrete  the initial value problems to construct new type iterative regularization methods.
\end{abstract}

\maketitle

\section{Introduction}
Consider the ill-posed operator equation
\begin{equation}\label{1.1}
F(x) = y\,,
\end{equation}
where $F: \mathcal{D}(F)\subset \X\rightarrow \Y$ is a nonlinear operator, $\X,\Y$ are infinite dimensional  Hilbert spaces.
Instead of exact data $y$, in practice we are given only noisy data $y^\d$ satisfying
\begin{equation}\label{1.2}
\|y-y^\d\|\leq \d
\end{equation}
with known noise level $\d$. 

In order to find the solution of \eqref{1.1} with desired feature such as sparsity and piecewise constancy,
 one may introduce non-smooth penalty functions, such as the $L^1$ and total variation like
functions \cite{bt09,bt09(2)}. In recent years, the relative iterative regularization has received tremendous attention,
while
Landweber type iteration is one of the most fundamental methods \cite{bh12,jw13}. 

For a proper, lower semi-continuous,
uniformly convex function as penalty term $\Theta: \X \to (-\infty,+\infty) $ given the initial pairs $(\xi_0^\d,x_0^\d) = (\xi_0,\nabla \Theta^*(\xi_0))$, the Landweber iteration with penalty $\Theta$ takes
the form
  \begin{equation}\label{land}
  \begin{aligned}
\xi_{n+1}^\d&=\xi_n^\d-\mu_n^\d F'(x_n^\d)^*(F(x_n^\d)-y^\d)\,,\\
x_{n+1}^\d &= \nabla \Theta^*( \xi_{n+1}^\d)\,,\quad n=0,1,2,\ldots
\end{aligned}
  \end{equation}
where $F'(\cdot)^*$ denotes the adjoint of $F'$, $\mu_n^\d$ is the step-size and $\Theta^*:\X\to (-\infty, \infty]$ is the Legendre-Fenchel conjugate of $\Theta$.
The advantage of this method is the freedom on the choice of $\Theta$  so that the method can be utilized in detecting different features of the sought solution.  Extensive researches have been carried out to study above Landweber-type iteration, its  accelerated version and its modification. The  convergence analysis  of the iterative methods has been given under discrepancy principle or Heuristic choice rule \cite{jin16,ms16,ghc19,zwj19,rj20,zw2020,zqw21}.

In this paper, we investigate the methods of asymptotical regularization, which can be considered as the continuous analogue of iteration \eqref{land}. By introducing an artificial time variable, a regularized approximation pair $(\xi^\d(T),x^\d(T))$
is obtained by solving the initial value problem
 \begin{equation}\label{asym1}
  \begin{aligned}
\frac{d\xi^\d(t)}{dt}& = - F'(x^\d(t))^*(F(x^\d(t))-y^\d), \quad 0<t\leq T, \\
x^\d(t) &= \nabla \Theta^*(\xi^\d(t)),\\
\end{aligned}
  \end{equation}
  with initial value $(\xi^\d(0), x^\d(0)) = (\xi_0,\nabla \Theta^*(\xi_0))$.
  Note that, when take $\Theta(x) = \frac{1}{2}\|x\|^2$, then \eqref{asym1} became the asymptotical regularization in \cite{t94}
   \begin{equation}\label{asym0}
  \begin{aligned}
\frac{dx^\d(t)}{dt}& = - F'(x^\d(t))^*(F(x^\d(t))-y^\d), \quad 0<t\leq T. \\
\end{aligned}
  \end{equation}
with initial value $x^\delta(0) = x_0.$
In \cite{t94}, Ulrich Tautenhahn
 give convergence properties of method \eqref{asym0} and derive stability estimates. It also has been
show that $T$ plays the role of the regularization parameter and can be chosen by a generalized discrepancy principle.
For the autonomous ODE \eqref{asym0},  Runge-Kutta method can be used  and the   new type iterative regularization methods can be  constructed. The theory of the RK regularization method has been well-developed for linear inverse problems \cite{r05,zml20}, and for nonlinear inverse problems \cite{lhw07,bp08}, as some special cases, such as explicit Euler method(Landweber iteration)\cite{hns95},  explicit 2-stage R-K Landweber method \cite{lhw07}, implicit Euler method (Levenberg-Marquardt iteration) \cite{h97}.

Recently, there has been a great interest in analysing such dynamical flows. In \cite{zml20}, variational source conditions have been established for general spectral regularization methods which yield convergence rates. In \cite{lnw20}, the authors interpreted steady linear statistical inverse problems with white noise, and introduced a stochastic differential equation system. In \cite{bdes21}, the authors applied the general theory of convergence rates to the three, well studied examples, including Showalter's method, heavy ball method and the vanishing viscosity method.
Besides general spectral analysis and source condition, the conditional stability assumption was also a efficient tool to analyze the convergence rate, which was fist proposed in \cite{cy00}, where the convergence rate of the Tikhonov regularization in Hilbert spaces was proved, then, the results have been extended to Banach spaces in \cite{chl14}. In \cite{dqs14}, the authors analyzed a nonlinear Landweber iteration and proved its local convergence rates under the H\"{o}lder continuity of the inverse mapping. The latest related work was \cite{zqw21}, the authors discussed the convergence rate of Landweber-type iteration in Banach spaces.

In this manuscript,
we will analyze  the regularization property of  asymptotical regularization \eqref{asym1} through a proper choice of the terminating time. 
The paper is organized as follows. In section 2, we give some preliminaries from convex analysis.
In section 3, we present  the detailed analysis
of convergence and regularization properties of the  asymptotical regularization \eqref{asym1} including the convergence rate under
 conditional stability assumption.

\section{Tools from convex analysis}
\setcounter{equation}{0}
Given a convex function $\Theta: \X \to (-\infty, \infty]$, we use $ \p \Theta(x)$ to denote the subdifferential
of $\Theta$  at $x\in \X$, i.e.
$$
\p \Theta(x) := \{\xi \in \X: \Theta(\bar x) -\Theta(x) -\l \xi, \bar x-x\r\ge 0 \mbox{ for all } \bar x\in \X\}.
$$
Let $\D(\Theta): = \{x\in \X:\Theta(x)<\infty\}$ be its effective domain and let
\begin{equation*}
\D(\p\Theta): = \{x\in \D(\Theta):\p\Theta(x)\neq\varnothing\}.
\end{equation*}
 The Bregman distance induced by $\Theta$ at $x$ in the direction $\xi\in \p\Theta(x) $  is defined by
\begin{equation*}
D_\xi \Theta(\bar x,x):=\Theta(\bar x)-\Theta(x)-\l \xi, \bar x-x\r, \qquad \forall \bar x\in  \X
\end{equation*}
which is always nonnegative and satisfies the identity
\begin{equation}\label{2.1}
D_{\xi_2} \Theta(x,x_2)-D_{\xi_1} \Theta(x, x_1) =D_{\xi_2} \Theta(x_1,x_2) +\l \xi_2-\xi_1, x_1-x\r
\end{equation}
for all $x\in \D(\Theta), x_1, x_2\in \D(\p \Theta)$, and $\xi_1\in \p \Theta(x_1)$, $\xi_2\in \p \Theta(x_2)$.

A proper convex function $\Theta: \X \to (-\infty, \infty]$ is called uniformly convex if there exists
a strictly increasing function $h:[0,\infty)\rightarrow [0,\infty)$ with $h(0) = 0$ such that
\begin{equation} \label{2.2}
\Theta(\la \bar x+(1-\la)x) +\la (1-\la)h(\|x-\bar x\|) \le \la \Theta(\bar x) +(1-\la) \Theta(x)
\end{equation}
for all $\bar x,x\in\X$ and $\la\in[0,1]$. If $h(t) = c_0t^p$ for some $c_0>0$ and $p>1$ in (\ref{2.2}),
then $\Theta$ is called $p$-convex.

For a proper lower semi-continuous convex function $\Theta: \X \to (-\infty, \infty]$, its Legendre-Fenchel conjugate is defined by
\begin{equation*}
\Theta^*(\xi):=\sup_{x\in \X} \left\{\l\xi, x\r -\Theta(x)\right\}, \quad \xi\in \X
\end{equation*}
which is also proper, lower semi-continuous, and convex. If $\X$ is reflexive, then
\begin{equation}\label{2.3}
\xi\in \p \Theta(x) \Longleftrightarrow x\in \p \Theta^*(\xi) \Longleftrightarrow \Theta(x) +\Theta^*(\xi) =\l \xi, x\r.
\end{equation}
For detailed introduction of convex analysis, we refer to \cite{ZA02}, here we list some properties.
\begin{proposition}\label{Prop_pre}
For a proper lower semi-continuous convex function $\Theta$ and $p>1$, the following statements hold true:
 \begin{enumerate}
 \item The $\Theta$ is $p$-convex if and only if
\begin{equation}\label{pconvex}
D_\xi \Theta(\bar x,x)\geq c_0\|x-\bar x\|^p,\quad\forall \bar x\in \D(\Theta), \ x\in \D(\p\Theta), \ \xi\in\p\Theta(x)\,.
\end{equation}
\item If $\Theta$ is $p$-convex with $p\geq 2$, then $\D(\Theta^*)=\X$,
$\Theta^*$ is Fr\'{e}chet differentiable and its gradient $\nabla \Theta^*:  \X\to \X$ satisfies
\begin{equation*}
\|\nabla \Theta^*(\xi_1)-\nabla \Theta^*(\xi_2) \|\le \left(\frac{\|\xi_1-\xi_2\|}{2c_0}\right)^{\frac{1}{p-1}},
\quad \forall \xi_1, \xi_2\in \X.
\end{equation*}
\item If $\X$ is reflexive and $\Theta$ is $p$-convex with $p\geq 2$, then
\begin{equation*}
x=\nabla \Theta^*(\xi) \Longleftrightarrow \xi \in \p \Theta(x) \Longleftrightarrow x =\arg \min_{z\in \X} \left\{ \Theta(z) -\l \xi, z\r\right\}.
\end{equation*}
\item If $\Theta$ is $p$-convex, then for any pairs $(x,\xi)$ and $(\bar x,\bar \xi) $
with $x, \bar x \in \D(\p\Theta), \xi\in\p\Theta(x), \bar \xi \in\p\Theta(\bar{x})$, then
\begin{equation}
D_\xi \Theta(\bar x,x) \le \frac{1}{p^*(2c_0)^{p*-1}}\|\xi-\bar \xi\|^{p*}\,.
\end{equation}
where $p^*  = p/(p-1)$ be the conjugate exponent of $p$.
 \end{enumerate}
\end{proposition}

\begin{assumption}\label{A2}
During this manuscript, we provide the following assumptions
\begin{enumerate}
\item[(a)] Let $\Theta:\X\to (-\infty, \infty]$ be a proper, weak lower semi-continuous, and $p$-convex functional with $p>1$ such that the condition
\eqref{pconvex} is satisfied for some $c_0>0$.
\item[(b)] There exists $\rho > 0$, $x_0\in \X$ and $\xi_0\in \p\Theta(x_0)$ such that $B_{2\rho}(x_0)\subset \D(F)$, where $B_{2\rho}(x_0)$ denotes the closed ball around $x_0$ with radius $2\rho$. In addition, assume (\ref{1.1}) has a solution $x^*\in \D(\Theta)$ with
$$
D_{\xi_0} \Theta(x^*, x_0) \le c_0 \rho^p\,.
$$

\item[(c)] The operator $F$ is continuous and weakly closed on $\D(F)$.

\item[(d)]  There exists a family of bounded linear operators $\{L(x): \X \to \Y\}_{x\in B_{2\rho}(x_0)\cap \D(\Theta)}$ such that
$x\to L(x)$ is continuous on $B_{2\rho}(x_0)\cap \D(\Theta)$ and there is $0\le \eta <1$ such that tangential cone condition
\begin{equation*}
\|F(x)-F(\bar x) -L(\bar x) (x-\bar x)\|<\eta \|F(x) -F(\bar x)\|
\end{equation*}
holds true for all $x, \bar x \in B_{2\rho}(x_0)\cap \D(\Theta)$. Moreover, there is a constant $C_0>0$ such that
$$
\|L(x) \|_{\X\to \Y} \le C_0, \quad \forall x \in B_{2\rho}(x_0).
$$
\end{enumerate}
\end{assumption}

\section{Asymptotical regularization with convex penalty}
\setcounter{equation}{0}
In this paper, we investigate the method of asymptotical regularization, which is the continuous analogue of \eqref{land}. In this section a regularized approximation $x^\d(T)$ of $x^*$
is obtained by solving the coupling initial value problem
 \begin{equation}\label{asymp}
  \begin{aligned}
\frac{d\xi^\d(t)}{dt}& = - L(x^\d(t))^*(F(x^\d(t))-\yd), \quad 0<t\leq T, \\
x(t) &= \nabla \Theta^*(\xi^\d(t)),\\
x^\d(0) &= x_0,\quad \xi^\d(0)  = \partial\Theta(x_0)\,.
\end{aligned}
  \end{equation}
Here and subsequently, the $L(x) $ is a family of bounded
linear operators from $\X$ to $\Y$ satisfying Assumption \ref{A2} (d) and functional $\Theta$ satisfying  Assumption \ref{A2} (a). The $T$ plays the role of the regularization parameter.


\subsection{Convergence results}
In this subsection, we study some convergence properties of \eqref{asymp}, the basic principle for proof have been developed in \cite{t94}, where corresponding results have been proved for classical Landweber asymptotical regularization in Hilbert space.
\begin{lemma}\label{lem1}
Let Assumption \ref{A2} hold true. Let $(\xi^\d(T),x^\d(T))$  be the solution of \eqref{asymp} for $T>0$. Then the approximated solution $x^\d(T)$ and the residual $\|F(x^\d(T))-y^\d\|$ are continuous with respect to $T$.
\end{lemma}
\begin{proof}
Referring to Proposition 2.1 (ii),
\begin{equation*}
    \begin{aligned}
 \lim_{\Delta T \to 0} \left\|x^\d(T+\Delta T) - x^\d(T)\right\|&=\lim_{\Delta T \to 0}  \left\|\nabla \Theta^*(\xi^\d(T+\Delta T)) - \nabla \Theta^*(\xi^\d(T))\right\| \\
&\leq \lim_{\Delta T \to 0}  \left(\frac{\|\xi^\d(T+\Delta T)-\xi^\d(T)\|}{2c_0}\right)^{\frac{1}{p-1}} \\
& = \lim_{\Delta T \to 0}\left(\left\|\frac{\xi^\d(T+\Delta T)-\xi^\d(T)}{\Delta T}\right\|\cdot\frac{\Delta T}{2c_0}\right)^{\frac{1}{p-1}}\,.
\end{aligned}
\end{equation*}
Since
\begin{equation} \label{xidiff}
\lim_{\Delta T \to 0}\left\|\frac{\xi^\d(T+\Delta T)-\xi^\d(T)}{\Delta T}\right\| = \frac{d\xi^\delta(t)}{dt}|_{t=T}
\end{equation}
exists and $\lim_{\Delta T \to 0} \frac{\Delta T}{2c_0} = 0$. This implies $x^\d(T)$ is continuous. With the continuity of $F$ the residual is continuous as well.
\end{proof}

\begin{proposition}\label{prop2}
Let Assumption \ref{A2} hold true and $(\xi^\d(T),x^\d(T))$  be the solution of \eqref{asymp} for $T>0$. Let $\hat x$ be any solution of \eqref{1.1} satisfying $D_{\xi_0}\Theta(\hat x, x_0)\leq c_0\rho^p$.
Define
\begin{equation}\label{varphi}
\varphi(T): = D_{\xi^\d(T)}\Theta(\hat x,x^\d(T)).
\end{equation}
If $x^\d(T)\in B_{2\rho}(x_0)$, then $\varphi(T)$ is differentiable and
\begin{equation}
\varphi'(T)\leq -(1-\eta)\|F(x^\d(T))-y\|^2 + \d(1+\eta)\|F(x^\d(T))-y^\d\|\,.
\end{equation}
In the case $\d =0$,
\begin{equation}\label{Fxd0}
\int_0^\infty \|F(x(\tau))-y\|^2 d\tau < \frac{1}{1-\eta}  D_{\xi_0}\Theta(\hat x,x_0)\,,
\end{equation}
and
\begin{equation}\label{convergence}
\lim_{T\rightarrow\infty}\|F(x(T))-y\|=0\,.
\end{equation}
\end{proposition}
\begin{proof}
Referring to \eqref{2.1}, we have
\begin{equation*}
    \begin{aligned}
&\quad\lim_{\Delta T \to 0}\frac{\varphi(T+\Delta T)-\varphi(T)}{\Delta T}\\
&= \lim_{\Delta T \to 0}\frac{D_{\xi^\d(T+\Delta T)}\Theta(\hat x,x^\d(T+\Delta T))-D_{\xi^\d(T)}\Theta(\hat x,x^\d(T))}{\Delta T}\\
  &= \lim_{\Delta T \to 0}\frac{D_{\xi^\d(T+\Delta T)}\Theta(x^\d(T),x^\d(T+\Delta T))+\l \xi^\d(T+\Delta T)-\xi^\d(T),x^\d(T)-\hat x\r}{\Delta T}\,.
\end{aligned}
  \end{equation*}

The application of Proposition \ref{Prop_pre}(iv) provides the estimate of the first term
\begin{equation*}
    \begin{aligned}
&\quad\lim_{\Delta T \to 0} \frac{D_{\xi^\d(T+\Delta T)}\Theta(x^\d(T),x^\d(T+\Delta T))}{\Delta T}\\
 &= \lim_{\Delta T \to 0}\frac{(\Delta T)^{p^*-1}}{p^*(2c_0)^{p^*-1}}\left\|\frac{\xi^\d(T+\Delta T)-\xi^\d(T)}{\Delta T}\right\|^{p^*}\,.
\end{aligned}
  \end{equation*}
Referring to \eqref{xidiff} again, the above estimate converges to $0$.  For the second term
\begin{equation*}
    \begin{aligned}
&\quad\lim_{\Delta T \to 0}\left\l\frac{ \xi^\d(T+\Delta T)-\xi^\d(T)}{\Delta T},x^\d(T)-\hat x\right\r\\
&= \left\l \frac{d\xi^\d(T)}{dT},x^\d(T)-\hat x\right\r\\
& = -\left\l L(x^\d(T))^*(F(x^\d(T))-\yd),x^\d(T)-\hat x\right\r\\
& = -\left\l (F(x^\d(T))-\yd),L(x^\d(T))(x^\d(T)-\hat x)\right\r\\
& = -\|F(x^\d(T))-\yd\|^2+\left\l F(x^\d(T))-\yd,F(x^\d(T))-\yd-L(x^\d(T))(x^\d(T)-\hat x)\right\r\\
& = -\|F(x^\d(T))-\yd\|^2+\d\|F(x^\d(T))-y^\d\|+
\l F(x^\d(T))-\yd, y-F(x^\d(T))-L(x^\d(T))(\hat x-x^\d(T))\r\,.
\end{aligned}
  \end{equation*}
Since, assume $x^\d(T)\in B_{2\rho}(x_0)$, we have
\begin{equation}\label{instead}
\|y-F(x^\d(T))-L(x^\d(T))(\hat x-x^\d(T))\|\leq \eta\|y-F(x^\d(T))\|.
\end{equation}
Putting it into above estimates,  it is consequently that $\varphi(t)$ is differentiable and
\begin{equation*}
    \begin{aligned}
\varphi'(T) &= \lim_{\Delta T \to 0}\frac{\varphi(T+\Delta T)-\varphi(T)}{\Delta T}\\
&\leq -\|F(x^\d(T))-\yd\|^2+\d\|F(x^\d(T))-y^\d\|+\eta\|F(x^\d(T))-\yd\|\left(\d+\|F(x^\d(T))-y^\d\|\right)\\
& = -(1-\eta)\|F(x^\d(T))-\yd\|^2 + \d(1+\eta)\|F(x^\d(T))-y^\d\|\,.
\end{aligned}
  \end{equation*}

In the case $\d = 0$, by integration from $0$ to $\infty$, we have
\begin{equation*}
(1-\eta)\int_0^\infty\|F(x^\d(\tau))-y\|^2d\tau\leq -\int_0^\infty \varphi'(\tau)d\tau
\end{equation*}
which gives \eqref{Fxd0} and consequently \eqref{convergence}.
  \end{proof}

Proposition \ref{prop2} shows that, the Bregman distance $D_{\xi^\d(T)}\Theta(\hat x,x^\d(T))$ as a function of $T$ is monotonically decreasing, as far as $\|F(x^\d(T))-y^\d\|>\tau \d $ is holding with $\tau  = (1+\eta)/(1-\eta)$. Hence, it makes sense to choose the regularization parameter by following discrepancy principle, i.e., $T = T^*$ is a solution of the nonlinear equation
  \begin{equation}\label{h}
  h(T) = \|F(x^\d(T))-y^\d\|-\tau \d =0\,\ \textrm{with }\  \tau  > (1+\eta)/(1-\eta)\,.
  \end{equation}
In the following proposition, we prove equation \eqref{h} has a solution.
  \begin{proposition}\label{prop3}
  Under the condition of Proposition 3.2, if $\|F(x_0)-y^\d\|>\tau \d$, then the nonlinear equation $h(T)=0$ exists at least one solution $T = T^*$.
  \end{proposition}
  \begin{proof}
Referring to Lemma \ref{lem1}, the function $h(T)$ is continuous, and $h(0) =\|F(x_0)-y^\d\|-\tau \d>0$. It is sufficient to show there exists a $\bar{T}$, such that $h(\bar T)\leq 0$. In contrast, we assume $h(T)> 0$ for all $T<\infty$, i.e.,
 \[ \|F(x^\d(T))-y^\d\|>\tau \d, \quad {\rm for~~ all}\quad  T\geq0. \]
  Combining this with Proposition \ref{prop2}, for all $T\geq0$,
\begin{equation*}
    \begin{aligned}
 \varphi'(T) &\leq -(1-\eta)\|F(x^\d(T))-y\|^2 + \d(1+\eta)\|F(x^\d(T))-y^\d\|\\
 &<- \left(1-\eta-\frac{1+\eta}{\tau}\right)\|F(x^\d(t))-y^\d\|^2\,.
\end{aligned}
  \end{equation*}
By integration, it follows that,
   \begin{equation}
    \begin{aligned}
\varphi(T)-\varphi(0) \leq -\left(1-\eta-\frac{1+\eta}{\tau}\right)   \int_0^{T} \|F(x^\d(\tau))-y^\d\|^2 d\tau\,,
\end{aligned}
  \end{equation}
this implies
    \begin{equation}
    \begin{aligned}
 \left(1-\eta-\frac{1+\eta}{\tau}\right)   \int_0^{T} \|F(x^\d(t))-y^\d\|^2 dt \leq \varphi(0).
\end{aligned}
  \end{equation}
Thus,
 \[ \left(1-\eta-\frac{1+\eta}{\tau}\right)T\tau^2\delta^2< D_{\xi_0}\Theta(\hat x,x_0)\leq c_0\rho^p\,.\]
 This contradicts with the assumption when $T$ is large.
  \end{proof}

\begin{remark}
Since $\varphi'(T)\leq 0$ for $T\leq T^*$, we obtain that
\begin{equation*}
D_{\xi^\delta(T)}\Theta(\hat x,x^\delta(T))\leq D_{\xi_0}\Theta(\hat x,x_0)\leq c_0\rho^p\,.
\end{equation*}
Applying $p-$convexity of $\Theta$, it is consequently
\begin{equation*}
\|\hat x-x^\delta(T)\|\leq\rho\,.
\end{equation*}
Then, referring to triangle inequality, we find the solution $x^\delta(T)$ remains in $B_{2\rho}(x_0)\subset D(F)$. Furthermore, in the case $\d=0$, $x(T)$ remains in $B_{2\rho}(x_0)$ for all $T<\infty$, this is because $\frac{d}{dT} D_{\xi(T)}\Theta(\hat x,x(T))<0$ for all $T<\infty$.
\end{remark}

In the next, we prove that the solution $x(T)$ of \eqref{asymp} converges (for $T\rightarrow\infty$) to a solution of \eqref{1.1}.
  \begin{theorem}\label{th3.5}
Let Assumption \ref{A2} hold true and let $(\xi(T),x(T))$  be the solution of \eqref{asymp} in the noisefree case for $T>0$. Then, there exists a solution $\bar x\in B_{2\rho}(x_0)\cap D(\Theta)$ of (\ref{1.1}) such that
\begin{equation*}
\lim_{T\rightarrow \infty} \|x(T)-\bar x\|=0 
\end{equation*}
\end{theorem}

\begin{proof}
We define function $\theta(T)$  such that
  \[
\|F(x(\theta(T)))-y\| = \min_{0\leq t\leq T} \|F(x(t))-y\|\,.
\]
It is obvious that the function $\theta(T)$ is continuous, monotonically nondecreasing and can be bounded by $T$. Moreover $\theta(T)\rightarrow\infty$ as $T\rightarrow\infty$. First, we prove the sequence $x(\theta(T))$ converges as $T\rightarrow\infty$. To this end, we will prove that for any $\varepsilon>0$, there exists $T_0>0$, such that for arbitrary $T_2>T_1>T$, there holds
\[
D_{\xi(\theta(T_1))}\Theta(x(\theta(T_2)), x(\theta(T_1)))< \varepsilon\,.
\]
Let $\hat x$ be any solution of \eqref{1.1} satisfying $D_{\xi_0}\Theta(\hat x,x_0)\leq c_0\rho^p$, referring to \eqref{2.1}, we can rewrite the term $D_{\xi(\theta(T_1))}\Theta(x(\theta(T_2)), x(\theta(T_1)))$ as
\begin{equation}\label{rewrite}
  D_{\xi(\theta(T_2))}\Theta(\hat{x}, x(\theta(T_2))) -D_{\theta(T_1)}\Theta(\hat{x}, \theta(T_1))
+\l \xi(\theta(T_1))-\xi(\theta(T_2)), x(\theta(T_2))-\hat{x}\r\,.
\end{equation}
Since $\{D_{\xi(T)}\Theta(\hat x, x(T))\}$ is strictly decreasing for $T>0$, there exists a constant $a\geq 0 $ and $\overline{T}>0$, such that for all $T>\overline{T}$, the first and second terms can be bounded as
\[a - \frac{\varepsilon}{3}< D_{\xi(\theta(T))}\Theta(\hat{x}, x(\theta(T))) <  a + \frac{\varepsilon}{3}\,.\]
For the third term of \eqref{rewrite}, the fundamental theorem of calculus provides
\begin{equation}
  \begin{aligned}
\left|\l \xi(\theta(T_1))-\xi(\theta(T_2)), x(\theta(T_2))-\hat{x}\r\right|
&\leq \left|\left\l\int_{\theta(T_1)}^{\theta(T_2)} L(x(s))^*(y-F(x(s)))ds,x(\theta(T_2))-\hat x\right\r\right|\\
&\le \int_{\theta(T_1)}^{\theta(T_2)}  \|y-F(x(s))\|\|L(x(s))(x(\theta(T_2))-\hat x)\|ds\,.
\end{aligned}
  \end{equation}
Applying  Assumption \eqref{A2} (c), and the  definition of the function $\theta(T)$, it is easy to obtain
  \begin{equation}
  \begin{aligned}
 \|L(x(s))(x(\theta(T_2))-\hat x)\|&= \|L(x(s))(x(\theta(T_2))-x(s))\|+ \|L(x(s))(x(s)-\hat x)\|\\
 &\leq3(1+\eta)\|y-F(x(s))\|\,.
\end{aligned}
  \end{equation}
Therefore, combining the property $\lim_{T\to \infty} \|F(x(T))-y\| = 0$, there exists a constant (still denote as) $\overline{T}$, for $T_2>T_1>\overline{T}$,
\begin{equation}
  \begin{aligned}
\left|\l \xi(\theta(T_1))-\xi(\theta(T_2)), x(\theta(T_2))-\hat{x}\r\right|
\leq 3(1+\eta)\int_{\theta(T_1)}^{\theta(T_2)} \|y-F(x(s))\|^2 ds
 < \frac{\varepsilon}{3}\,,
\end{aligned}
  \end{equation}
and consequently,
\begin{equation*}
D_{\xi(\theta(T_1))}\Theta(x(\theta(T_2)), x(\theta(T_1)))< (a+\frac{\varepsilon}{3})-(a-\frac{\varepsilon}{3}) +\frac{\varepsilon}{3} =\varepsilon\,.
\end{equation*}
The application of the $p$-convexity of  $\Theta$ implies $\lim_{T\to \infty} x(\theta(T))$ exists, we denote it as $\bar x$.

Similarly, we consider
 \begin{equation}
  \begin{aligned}
 &D_{\xi(\theta(T))}\Theta(x(\theta(T)), x(T))\\
 & =D_{\xi(\theta(T))}\Theta(\hat{x}, x(\theta(T))) -D_{\xi(T)}\Theta(\hat{x}, x(T))
+\l \xi(T)-\xi(\theta(T)), x(\theta(T))-\hat{x}\r
\end{aligned}
  \end{equation}
Since $\theta(T)\leq T$ and $ \|F(x(\theta(T)))-y\|\leq \|F(x(T))-y\|$, we have
 \begin{equation}
  \begin{aligned}
\left|\l \xi(T)-\xi(\theta(T)), x(\theta(T))-\hat{x}\r\right|
&\le \left|\left\l\int_{\theta(T)}^{T} L(x(s))^*(F(x(s))-y)ds,x(\theta(T))-\hat x\right\r\right|\\
&\le \int_{\theta(T)}^{T}  \|F(x(s))-y\|\|L(x(s))(x(\theta(T))-\hat x)\|ds\\
&\le 3(1+\eta)\int_{\theta(T)}^{T}\|y-F(x(s))\|^2 ds.
\end{aligned}
  \end{equation}
The same technique provides
\begin{equation*}
\lim_{T\to\infty}\left|\l \xi(\theta(T))-\xi(T), x(\theta(T))-\hat{x}\r\right| = 0
\end{equation*}
and consequently,
\begin{equation*}
\lim_{T\to \infty}D_{\xi(\theta(T))}\Theta(x(\theta(T)), x(T))\to 0
\end{equation*}
and $\lim_{T\to \infty}\| x(T)-x(\theta(T))\| = 0$. The application of triangle inequality gives
  \[
  \|x(T)-\bar x\| \leq \|x(\theta(T))-\bar x\|+\| x(T)-x(\theta(T))\|,
 \]
therefore,
  \[
 \lim_{T\to \infty} x(T) =\bar x.
  \]
By the weakly compactness of $F$ and continuity of $x(T)$, we can conclude $\bar x$ is a solution to \eqref{1.1} and $\bar x\in B_{2\rho}(x_0)\bigcap D(\Theta)$.
  \end{proof}

Finally, we perform the
analysis from the H\"{o}lder continuity of the inverse mapping instead of the source and non-linearity conditions and obtain the convergence rates.
\begin{assumption}\label{Holder}
Assume the H\"{o}lder type conditional stability, i.e., denote the $\nu\in[\frac{p}{2},p]$, there exists a constant $R_F$ such that
\begin{equation*}
D_{\bar\xi}\Theta(x,\bar x)\leq R_F \|F(x)-F(\bar x)\|^\nu,\quad \forall x,\bar x\in B_{2\rho}(x_0), \bar x\in \mathcal{D}(\partial\Theta), \bar\xi\in\partial\Theta(\bar x)\,.
\end{equation*}
\end{assumption}
\begin{theorem}\label{regularization}
Let Assumption \ref{A2}, in which the tangential cone condition is replaced by Lipschitz continuity
  of $F'$
     \begin{equation}\label{LipF}
    \|F'(x)-F'(\bar x)\|\leq L'\|x-\bar x\|,\quad x,\bar x\in B_{2\rho}(x_0).
    \end{equation}
 and Assumption \ref{Holder} hold true. Assume $\|F(x_0)-y^\d\|>\tau \d$, let $(\xi^\d(T),x^\d(T))$  be the solution of \eqref{asymp} with $L(x)=F'(x), $ where $T=T^*$ is the solution of \eqref{h}. Then, there is a solution $\bar x$  of \eqref{1.1} in $B_{2\rho}(x_0)$, there holds
\begin{equation*}
D_{\xi^\d(T^*)}^\d\Theta(\bar x,x^\d(T^*))\to 0, \quad \mbox{as } \d\rightarrow 0\,,
\end{equation*}
and the convergence rate is given by
\begin{equation*}
D_{\xi^\d(T^*)}^\d\Theta(\bar x,x^\d(T^*))\leq R_F(\tau+1)^\nu\d^\nu\,.
\end{equation*}
\end{theorem}
\begin{proof}
Assume $x^\d(T)\in B_{2\rho}(x_0)$ for $T<T^*$, we have
 \begin{equation}\label{instead}
  \begin{aligned}
\|y-F(x^\d(T))-F'(x^\d(T))(\hat x-x^\d(T))\|
&\leq\frac{L}{2}\|x^\d(T)-\hat x\|^2\nonumber\\
&\leq\frac{L}{2}\|x^\d(T)-\hat x\|^{2-\frac{p}{\nu}}\|x^\d(T)-\hat x\|^\frac{p}{\nu}\nonumber\\
&\leq\frac{L}{2}(4\rho)^{2-\frac{p}{\nu}}\left(\frac{1}{c_0}D_{\xi^\d(T_n)}\Theta(\hat x, x^\d(T_n))\right)^\frac{1}{\nu}\nonumber\\
&\leq\frac{L}{2}(4\rho)^{2-\frac{p}{\nu}}\left(\frac{R_F}{c_0}\right)^{\frac{1}{\nu}}\|y-F(x^\d(T_n))\|\,.
\end{aligned}
  \end{equation}
Denote
\begin{equation}\label{eta}
\rho = \frac{1}{4}\left(\frac{2\eta}{L}\right)^{\frac{\nu}{2\nu-p}}\left(\frac{c_0}{R_F}\right)^{\frac{1}{2\nu-p}}
\end{equation}
then Assumption \ref{Holder} with $p/2\leq\nu\leq p$  and \eqref{LipF}  imply the general tangential cone condition.
The conclusion is directly obtained by utilizing the conditional stability in Assumption \ref{Holder},
\begin{equation}
  \begin{aligned}
D_{\xi^\d(T^*)}^\d\Theta(\bar x,x^\d(T^*))&\leq R_F\|y-F(x^\d(T^*))\|^\nu\\
&\leq R_F\left(\d+\|y^\d-F(x^\d(T^*))\|\right)^\nu\\
&\leq R_F(\tau+1)^\nu\d^\nu\,.
\end{aligned}
  \end{equation}

\end{proof}

\begin{remark}
Note that one obtains for finite $T<\infty$, the existence of a unique solution $\xi(t) \in C(0, T;\X)$
of \eqref{asymp}  if the operator $G(\xi) = F'(\nabla\Theta^*(\xi))^*[y^\d - F(\nabla\Theta^*(\xi))]$ is locally Lipschitz continuous in $\X$   \cite{db02}.
Actually, for $(\xi_1,x_1)$ and $(\xi_2,x_2)$ with $x_i = \nabla\Theta^*(\xi_i)$, $i=1,2$, and $x_1,x_2 \in B_{2\rho}(x_0)$, we have
\begin{equation*}
\begin{aligned}
G(\xi_1)-G(\xi_2)&= F'(x_1)^*(y^\d - F(x_1))-F'(x_2)^*(y^\d - F(x_2))\\
 &= \left(F'(x_1)^*-F'(x_2)^*\right)\left(y^\d-F(x_1)\right)
 + F'(x_2)^*\left(F(x_2)-F(x_1) \right)dt,
\end{aligned}
\end{equation*}
Applying the  Lipschitz continuity of $F'$ \eqref{LipF} and the bounded property of $F'$, then we have
\begin{equation*}
\begin{aligned}
\|G(\xi_1)-G(\xi_2)\|&\leq  \|\left(F'(x_1)^*-F'(x_2)^*\right)\left(y^\d-F(x_1)\right)\|
 \| F'(x_2)^*\left(F(x_2)-F(x_1) \right)\|\\
& \leq  L'\|y^\d-F(x_1)\|\|x_1-x_2\|+C_0^2\|x_1-x_2\|
\end{aligned}
\end{equation*}
Take the $\Theta$ as strongly convex (p=2) then
\begin{equation*}
\|x_1-x_2\|=\|\nabla \Theta^*(\xi_1)-\nabla \Theta^*(\xi_2) \|\leq \frac{\|\xi_1-\xi_2\|}{2c_0}.
\quad \forall \xi_1, \xi_2\in \X.
\end{equation*}
Combining above estimates, we obtain that  $G(\xi)$ is locally Lipschitz continuous in $\X$, which implies that
the existence of a unique solution $\xi(t) \in C(0, T; \X)$
of \eqref{asymp}.
\end{remark}

\section{Discretization of asymptotical method}

In this section we recall the Runge-Kutta (RK) type iterative methods which are regarded as the discretization of asymptotical method.

RK methods are a class of one-step methods for the numerical solutions of the initial value problems for ODEs,
\[
\frac{dw(t)}{dt} = \Psi(t,w(t)), \quad t>0, \quad w(t) = w_0\,,
\]
with given $\Psi:[0,\infty]\times X \to X$ and $w_0\in X$.

A RK method of $s$-stages provides approximations $w_n$ to $w(t_n)$ with $t_n = \sum_{k=1}^n \triangle t_k$ such that

  \begin{equation}\label{RK1}
  \begin{aligned}
&w_{n+1} = w_{n}+\triangle t_n \sum_{i=1}^s b_i\Psi (t+c_i\triangle t_n,k_i),\\
&k_i = w_n + \triangle t_n \sum_{j=1}^s a_{i,j} \Psi(t+c_j \triangle t_n ,k_j).
\end{aligned}
  \end{equation}
The given coefficients $A = (a_{ij})\in \mathbb{R}^{s\times s}$, $b= (b_1,b_2,\cdots,b_s)^T$ and $c = (c_1,c_2,\cdots,c_s)^T$ determine the particular method. The method is called {\it explicit} if $A$ is strictly lower triangular matrix. Otherwise, the method is called {\it implicit}, since linear or nonlinear algebraic equations have to be solved to calculate $k_i, i=1,\cdots, s$. Usually, the specific instance of RK methods can be presented by Butcher tableau, see Table 1 or by triple $(A,b,c)$  in \cite{db02}. We only consider consistent $s-$stage RK method, i.e., $\sum_{i=1}^s b_i=1$. However, an inconsistent $s-$stage RK method can be considered in the same manner.
 \begin{table}
 \caption{The general Butcher tableau of RK method. }
\begin{center}
\begin{tabular}
{c|c}
c &A \\
\hline
       &  $b^{T}$\\
\end{tabular}
\qquad\qquad
\begin{tabular}{c|ccc}
        $c_1$& $a_{11}$ &$\cdots$ &$a_{1s}$\\
       $\vdots$&$\vdots$ &$\vdots$ &$\vdots$\\
        $c_s$& $a_{s1}$ &$\cdots$ &$a_{ss}$\\
 \hline
       \ &$b_{1}$ &$\cdots$ &$b_{s}$\\
\end{tabular}
\end{center}
\label{}
\end{table}

A RK method applied to autonomous ODE \eqref{asymp}, that is,
\begin{equation*}
\Psi(\xi^\d(t)) = L(x^\d(t))^*(y^\d-F(x^\d(t)))\,,\quad x^\d(t) = \nabla\Theta^*(\xi^\d(t))\,.
\end{equation*}
Then, it is consequently
  \begin{equation}
  \begin{aligned}
&\xi_{n+1}^\d = \xi_{n}^\d+\triangle t_n\sum_{i=1}^s b_iL(\nabla\Theta^*(k_i))^*(y^\d-F(\nabla\Theta^*(k_i)))\,,\\
&k_i = \xi_{n}^\d+ \triangle t_n\sum_{j=1}^s a_{ij}L(\nabla\Theta^*(k_j))^*(y^\d - F(\nabla\Theta^*(k_j)))\,.
\end{aligned}
  \end{equation}
Use the explicit Euler method, i.e., $s=1$, $c=0$, $b=1$, $A=0$, it follow that
 \begin{equation}
  \begin{aligned}
&\xi_{n+1}^\d = \xi_{n}^\d + \triangle t_n L(x_{n}^\d)^*(y^\d- F(x_{n}^\d)),\\
&x_{n+1}^\d  = \nabla\Theta(\xi_{n+1}^\d).
\end{aligned}
  \end{equation}
It is a Landweber-type iteration. Use the implicit Euler method with $s=1$, $c=1$, $b=1$ and $A=1$, we have the implicit Landweber-type iteration
 \begin{equation}
  \begin{aligned}
\xi_{n+1}^\d&=\xi_n^\d+\triangle t_n L(x_{n+1}^\d)^*(F(x_{n+1}^\d)-y^\d),\\
x_{n+1}^\d &= \nabla \Theta^*( \xi_{n+1}^\d).
\end{aligned}
  \end{equation}
For $s=2$, the explicit $2-$stage RK methods provides
 \begin{equation}
  \begin{aligned}
&k_2 = \xi^\d_n + \triangle t_n a_{21}L(x_n^\d)^*(y^\d -F(x_n^\d)),\\
&z_n^\d = \nabla\Theta^*(k_2),\\
&\xi_{n+1}^\d = \xi_{n}+\triangle t_n \left(b_1(y^\d-F(x_n^\d))+b_2L(z_n^\d)^*(y^\d-F(z_n^\d))\right), \\
&x_{n+1}^\d = \nabla \Theta^*( \xi_{n+1}^\d)\,,
\end{aligned}
  \end{equation}
in which the coefficients should satisfy
\begin{equation}
b_1+b_2=1\,,\quad a_{21}b_2 = \frac{1}{2}\,.
\end{equation}

\section*{\bf Acknowledgements}
 The work of M Zhong is supported by the National Natural Science Foundation of China (No. 11871149) and supported by Zhishan Youth Scholar Program of SEU. 
The work of W Wang is supported by the National Natural Science Foundation of China (No. 12071184).


\begin{thebibliography}{99}

 \bibitem{bdes21}
R. B\c{o}t, G. Dong, P.  Elbau and O. Scherzer, Convergence rates of first-and higher-order dynamics for solving linear ill-posed problems. Foundations of Computational Mathematics. 2021  https://doi.org/10.1007/s10208-021-09536-6.
\bibitem{bh12}
R. Bo\c{t} and T. Hein, Iterative regularization with a geeral penalty term: theory and applications
 to $L^1$ and TV regularization,  Inverse Problems   28  (2012) 104010.

\bibitem{bp08}
C. Bckmann, P. Pornsawad, Iterative Runge-Kutta-type methods for nonlinear ill-posed problems. Inverse Problems 24(2) (2008) 025002.



\bibitem{bt09}
A. Beck and M. Teboulle,    A fast iterative shrinkage-thresholding algorithm for linear inverse
problems, SIAM J. Imag. Sci.  2 (2009)  183--202.
\bibitem{bt09(2)}
A. Beck and M. Teboulle,   Fast gradient-based algorithms for constrained total variation image
denoising and deblurring problems, IEEE Trans. Image Process. 18 (2009)  2419--2434.
   \bibitem{chl14}
  J. Cheng, B. Hofmann and S. Lu,
  The index function and Tikhonov regularization for ill-posed problems, J. Comput. Appl. Math. 265  (2014) 110--119.

 \bibitem{cy00}
 J. Cheng  and Y. Yamamoto,
   One new strategy for a priori choice of regularizing parameters in Tikhonov's regularization,
  Inverse Problems 16 (2008) L31-L38.

\bibitem{db02}
P. Deuflhard and F. Bornemann, Scientific Computing with Ordinary Differential Equations, New York:
Springer, 2002.


\bibitem{ehn96}
H.  Engl, M. Hanke, A. Neubauer,  Regularization of Inverse Problems, New York (NY): Springer, 1996.

  \bibitem{ghc19}
R. Gu, B. Han and Y. Chen, Fast subspace optimization method for nonlinear inverse
problems in Banach spaces with uniformly convex penalty terms, Inverse Problems  35 (2019) 125011.
   \bibitem{dqs14}
  M. V. de Hoop, L. Qiu, O. Scherzer,
   Local analysis of inverse problems: H\"{o}lder stability
  and iterative reconstruction,
  Inverse Problems  28 (2012) 045001.



\bibitem{h97}
M. Hanke, A regularizing Levenberg Marquardt scheme with applications to inverse groundwater
filtration problems, Inverse Problems 13 (1997) 79--95.


\bibitem{hns95}
M. Hanke, A. Neubauer and O. Scherzer, A convergence analysis of the Landweber iteration
for nonlinear ill-posed problems, Numer. Math.  72 (1995)  21--37.
   \bibitem{jin15}
 Q. Jin, Inexact Newton-Landweber iteration in Banach spaces with nonsmooth convex penalty terms, SIAM J. Numer. Anal. 53 (5) (2015) 2389--2413.
   \bibitem{jin16}
  Q. Jin, Landweber-Kaczmarz method in Banach spaces with inexact inner solvers, Inverse
Problems  32 (2016) 104005.
   \bibitem{jw13}
  Q. Jin, and W. Wang,
    Landweber iteration of Kaczmarz type with general non-smooth convex penalty functionals,
  Inverse Problems 29 (2013) 085011.




 \bibitem{lhw07} L. Li, B. Han and W. Wang,  R-K type Landweber method for nonlinear ill-posed
problems, J. Comput. Appl. Math.   206 (2007) 341--357.
\bibitem{lnw20}
S. Lu,  P. Niu  and F. Werner, On the asymptotical regularization for linear inverse problems in presence of white noise,
SIAM/ASA Journal on Uncertainty Quantification  9(1) (2021) 1--28.


  \bibitem{ms16}
  P. Maa{\ss}  and R. Strehlow,
   An iterative regularization method for nonlinear problems based on Bregman projections,
   Inverse Problems
 32 (2016) 115013.

\bibitem{r99}
A. Rieder, On the regularization of nonlinear ill-posed problems via inexact Newton iterations,
Inverse Problems, 15 (1999)  309-327.
\bibitem{rof92}
L. Rudin, S. Osher and C. Fatemi,  Nonlinear total variation based noise removal algorithm, Phys. D   60 (1992) 259-268.

\bibitem{r05}
A. Rieder, Runge-Kutta integrators yield optimal regularization schemes, Inverse Problems 21 (2005) 453--471.
  \bibitem{rj20}
 R.  Real and Q. Jin,
   A revisit on Landweber iteration,
   Inverse Problems 36 (2020) 075011.
\bibitem{t94}
U. Tautenhahn, On the asymptotical regularization of nonlinear ill-posed problems, Inverse Problems
10 (1994) 1405-1418.


 \bibitem{ZA02}
  C. Z$\breve{a}$linscu,
   Convex Analysis in General Vector Spaces World Scientific Publishing Co., Inc., River
  Edge. 2002.
  \bibitem{zml20}
Y. Zhao, P. Math\'{e}, S. Lu, Convergence analysis of asymptotical regularization and Runge-Kutta integrators for linear inverse problems under variational source conditions, CSIAM Transactions on Applied Mathematics   1(4) (2020) 693--714.
  \bibitem{zqw21}
 M. Zhong, L. Qiu, W. Wang,  Landweber-type method with uniformly convex constraints  under conditional stability assumptions,
 preprint.

  \bibitem{zw2020}
M. Zhong, W. Wang, The two-point gradient methods for nonlinear inverse problems based
on Bregman projections, Inverse Problems  36 (2020) 045012. 

 \bibitem{zwj19}
  M. Zhong, W.  Wang and Q. Jin,
  Regularization of inverse problems by two-point gradient methods in Banach spaces,
   Numer. Math.  143  (2019) 713--747.
\end{thebibliography}
\end{document}